\documentclass[review]{elsarticle}

\usepackage{lineno,hyperref}
\usepackage{mathrsfs}
\usepackage{amsfonts}
\usepackage{amssymb}
\usepackage{amsmath}
\usepackage{hyperref}
\usepackage{amsthm}
\usepackage{cases}
\usepackage{arydshln}
\usepackage[usenames,dvipsnames]{color}
\usepackage{graphicx}%
\newtheorem{thm}{Theorem}[section]
\newtheorem{cor}[thm]{Corollary}
\newtheorem{lemma}[thm]{Lemma}

\newtheorem{rem}[thm]{Remark}

\modulolinenumbers[5]

\begin{document}

\begin{frontmatter}

\title{An improvement and generalization of  Rotfel'd type inequalities for sectorial matrices}

\author{Fanghong Nan\fnref{1footnote}}\author{Teng Zhang\fnref{2footnote}}
\address{School of Mathematics and Statistics, Xi'an Jiaotong University, Xi'an 710049, China}
\fntext[1footnote]{Email: nfh1108@stu.xjtu.edu.cn}
\fntext[2footnote]{Email: teng.zhang@stu.xjtu.edu.cn}

\begin{abstract} By using equivalence conditions for sectorial matrices obtained by Alakhrass and Sababheh in 2020, we improve a Rotfel'd type inequality for sectorial matrices derived by  P. Zhang in 2015 and generalize a result derived by Y. Mao et al. in 2024. 
\end{abstract}


\end{frontmatter}


\section{Introduction}
 Let $\mathbb{M}_{n}$ be the set of all $n \times n$ complex matrices. The set of $n \times n$ positive semidefinite matrices is denoted by $\mathbb{M}_{n}^{+}$. A norm $\|\cdot\|$ on $\mathbb{M}_{n}$ is unitarily invariant if $\Vert{UAV}\Vert = \Vert{A}\Vert$ for any $A \in \mathbb{M}_{n}$ and unitary matrices $U,V \in \mathbb{M}_{n}$. For $A \in \mathbb{M}_{n}$, we denote by $A^{*}$ and $|A| = (A^{*}A)^{\frac{1}{2}}$ the conjugate transpose and the modulus of $A$, respectively. We denote the $j$-th largest singular value of $A$ by $\sigma_{j}(A)$. If $A$ is Hermitian, then its eigenvalues of $A$ are real and its $j$-th largest eigenvalue is written as $\lambda_{j}(A)$. Note that $\sigma_{j}(A) = \lambda_{j}(|A|)$, $j = 1,\cdots,n$. For two Hermitian matrices $A,B \in \mathbb{M}_{n}$, we write $A \leq B$ to mean $B-A \in \mathbb{M}_{n}^{+}$.

For $A \in \mathbb{M}_{n}$, we can write
\begin{equation*}
A = \Re A+i\Im A,
\end{equation*}
where 
\begin{equation*}
\Re A = \frac{1}{2}(A+A^{*}),\quad \Im A = \frac{1}{2i}(A-A^{*}).\nonumber
\end{equation*}
This is called the Cartesian decomposition of $A$. 

The numerical range of $A \in \mathbb{M}_{n}$ is defined by
\begin{equation*}
W(A) = \{x^{*}Ax| x \in \mathbb{C}^{n},x^{*}x=1\}.
\end{equation*}

Sectorial matrices were defined in \cite{AP03} as an extension of the definition of the positive matrices. For $\alpha \in [0,\pi/2)$, we define a sector on the complex plane
\begin{equation*}
S_{\alpha} = \{z \in \mathbb{C}^{n}| \Re z \geq 0,|\Im z| \leq (\Re z)\tan\alpha \}.
\end{equation*}

A matrix $A \in \mathbb{M}_{n}$ whose numerical range is a subset of a sector $S_{\alpha}$, for some $\alpha \in [0,\pi/2)$, is called a sectorial matrix. Relevant studies of sectorial matrices can be found in Drury and Lin \cite{DL14}, Zhang \cite{ZF15} and references therein. Note that if $\alpha = 0$, then a sectorial matrix $A$ is a positive matrix. It is obvious, from the definition of sectorial matrices that if $W(A)\subseteq S_{\alpha}$ for some $\alpha$, then $\Re A \in \mathbb{M}_{n}^{+}$. 

Consider a partitioned matrix $A \in \mathbb{M}_{n}$ in the form
\begin{equation}\label{e1}
A = \begin{bmatrix}
	A_{11}&A_{12}\\
	A_{21}&A_{22}
\end{bmatrix}, \quad \text{where diagonal blocks } A_{11}\text{ and } A_{22}  \text{ are square}.
\end{equation}

Let $A,B$ be positive semidefinite matrices and let $f(x)$ be a non-negative concave function on $[0,\infty)$. In  \cite{RS69}, Rotfel'd proved a subadditivity result for concave functions of sums of singular values of operators. In \cite{LE11}, Lee proved the following result which is considered as an extension of the classic Rotfel'd theorem.
\begin{thm}\label{lt}
Let $A \in \mathbb{M}_{n}$ be positive semidefinite and be partitioned as in (\ref{e1}), and let $f : [0,\infty) \to [0,\infty)$ be a concave function. Then for any unitarily invariant norm $\Vert{\cdot}\Vert$ it holds
\begin{equation}
\Vert{f(A)}\Vert \leq \Vert{f(A_{11})}\Vert+\Vert{f(A_{22})}\Vert.\nonumber
\end{equation}
\end{thm}

In \cite{ZP15}, Zhang considered sectorial matrices and obtained the following extension of Lee's result. 
\begin{thm}\cite[Theorem 3.4]{ZP15}\label{zpt}
Let $f : [0,\infty) \to [0,\infty)$ be a concave function and let $A$ with $W(A) \subseteq S_{\alpha}$ for $\alpha \in [0,\pi/2)$ be partitioned as in (1). Then
\begin{equation*}
\Vert{f(|A|)}\Vert \leq \Vert{f(|A_{11}|)}\Vert+\Vert{f(|A_{22}|)}\Vert+2(\Vert{f(\tan(\alpha)|A_{11}|)}\Vert+\Vert{f(\tan(\alpha)|A_{22}|)}\Vert).
\end{equation*} 
\end{thm}
Zhang also obtained the following result:
\begin{thm}\cite[Corollary 3.3]{ZP15}\label{zpc}
Let $f : [0,\infty) \to [0,\infty)$ be a concave function and let $A$ with $W(A) \subseteq S_{\pi/4}$ be partitioned as in $(1)$. Then
\begin{equation*}
\Vert{f(|A|)}\Vert \leq 2(\Vert{f(\tfrac{\sqrt{2}}{2}|A_{11}|)}\Vert+\Vert{f(\tfrac{\sqrt{2}}{2}|A_{22}|)}\Vert).
\end{equation*} 
\end{thm}
Zhang \cite{ZP15} asked whether the coefficient $2$ in Theorem \ref{zpc} can be replaced by $1$. Several authors have considered this problem and given solutions under certain conditions. In \cite{HZ17},  Hou and Zhang gave an answer to the problem when $\Im A \in \mathbb{M}_{n}^{+}$. In \cite{ZN18}, when $A$ is normal, Zhao and Ni affirmed this problem. They proved that
\begin{thm}\cite[Theorem 3.4]{ZN18}
Let $A$ be partitioned as in (\ref{e1}). Suppose $f : [0,\infty) \to [0,\infty)$ is a concave function and let $A$ with $W(A) \subseteq S_{\alpha}$ for $\alpha \in [0,\pi/2)$ and $A$ is normal. and  Then
\begin{equation}
\Vert{f(|A|)}\Vert \leq \Vert{f(|A_{11}|)}\Vert+\Vert{f(|A_{22}|)}\Vert+\Vert{f(\tan(\alpha)|A_{11}|)}\Vert+\Vert{f(\tan(\alpha)|A_{22}|)}\Vert.\nonumber
\end{equation} 
\end{thm}
Recently, Yang, Lu and Chen \cite{YL19} improved Theorem \ref{}. They proved
\begin{thm}\cite[Theorem 3.1]{YL19}
Let $f : [0,\infty) \to [0,\infty)$ be a concave function and let $A$ with $W(A) \subseteq S_{\alpha}$ for $\alpha \in [0,\pi/2)$ and assume $A$ is normal, and let $A$ be partitioned as in $(1)$. Then
\begin{equation}
\Vert{f(|A|)}\Vert \leq \Vert{f(\sec\alpha|A_{11}|)}\Vert+\Vert{f(\sec\alpha|A_{22}|)}\Vert.\nonumber
\end{equation} 
\end{thm}
Fu and Liu \cite{FL16} also obtained the following generalization of Theorem \ref{lt}.
\begin{thm}\cite[Theorem 3]{FL16}
Let $f : [0,\infty) \to [0,\infty)$ be a concave function and let $A$ with $W(A) \subseteq S_{\alpha}$ for $\alpha \in [0,\pi/2)$ be partitioned as in $(1)$. Then
\begin{equation}
\Vert{f(|A|)}\Vert \leq \Vert{f(\sec^{2}\alpha|A_{11}|)}\Vert+\Vert{f(\sec^{2}\alpha|A_{22}|)}\Vert.
\end{equation} 
\end{thm}
In \cite{MJ24}, the authors gave a generalization of Theorem \ref{zpc} to a general angle $\alpha \in [0,\pi/2)$.
\begin{thm}\cite[Theorem 2.7]{MJ24}\label{t1.7}\label{ss}
Let $f : [0,\infty) \to [0,\infty)$ be a concave function and let $A$ with $W(A) \subseteq S_{\alpha}$ for $\alpha \in [0,\pi/2)$ be partitioned as in $(1)$. Then for any unitarily invariant norm $\Vert{\cdot}\Vert$ it holds
\begin{equation}
\Vert{f(|A|)}\Vert \leq 2\left(\Vert{f\left(\frac{\sec\alpha}{2}|A_{11}|\right)}\Vert+\Vert{f\left(\frac{\sec\alpha}{2}|A_{22}|\right)}\Vert \right).\nonumber
\end{equation}
\end{thm}
In this paper, we improve Theorem \ref{zpt} and give a generalization of Theorem \ref{t1.7}.

\section{Improvements of Rotfel'd theorem for sectorial matrices}

In this section, we first list a few lemmas that are useful to derive our main results.
\begin{lemma} \label{t2.1}
	Let $\begin{bmatrix}
		A&X\\
		X^*&B
	\end{bmatrix}$ partitioned into four blocks in $\mathbb{M}_n$ be positive semidefinite. Then exists some unitary matrix $U$,  for every $s>0$,
	\[	
\left|X\right|\le \dfrac{s}{2}U^*AU+\dfrac{1}{2s}B,
	\]
	and exists some unitary matrix $V$,  for every $s>0$,
	\[
	\left|X^*\right|\le \dfrac{s}{2}A+\dfrac{1}{2s}V^*BV.
	\]
\end{lemma}

\begin{proof}
	Let  the polar decomposition of $X^*$ be $X^* = U|X^*|$ and $s>0$.  Notice that
	\begin{eqnarray*}
		\begin{bmatrix}
			sI&-U^*\\
		\end{bmatrix}\begin{bmatrix}
			A&X\\
			X^*&B
		\end{bmatrix}\begin{bmatrix}
			sI\\
			-U
		\end{bmatrix}
		&=&s^2 A-sU^*X^*-sXU+U^*BU.\\
		&=&s^2 A-2s\left|X^*\right|+U^*BU\ge 0.
	\end{eqnarray*}
	That is,
\begin{eqnarray}\label{mr}
	\left|X^*\right|\le \dfrac{s}{2}A+\dfrac{1}{2s}U^*BU.
\end{eqnarray}
		Note that
	\[
		\begin{bmatrix}
		B&X^*\\
		X&A
	\end{bmatrix}=\begin{bmatrix}
		0&I\\
		I&0
	\end{bmatrix}	\begin{bmatrix}
	A&X\\
	X^*&B
\end{bmatrix}	\begin{bmatrix}
0&I\\
I&0
\end{bmatrix}\ge 0,
	\]
	Applying (\ref{mr}) to $\begin{bmatrix}
		B&X^*\\
		X&A
	\end{bmatrix}$ and replacing $s$ by $1/s$ yield the desired result.
\end{proof}

\begin{lemma}\cite[Theorem 2.2]{AS20}\label{3.1}
Let $A \in \mathbb{M}_{n}$ and $\alpha \in [0,\pi/2)$. Then the following relations are equivalent:
\begin{enumerate}[(1)]
\item $W(A) \subseteq S_{\alpha}$.

\item $\begin{bmatrix}
	sec(\alpha)\Re A&A^{*}\\
	A&sec(\alpha)\Re A
\end{bmatrix} \geq 0$.

\item $\begin{bmatrix}
	\tan(\alpha)\Re A&\Im A\\
	\Im A&tan(\alpha)\Re A
\end{bmatrix} \geq 0$.
\end{enumerate}
\end{lemma}
\begin{lemma}\cite[Theorem III.5.6]{BR97}\label{3.2}
Let $A$, $B \in \mathbb{M}_{n}$. Then there exist two unitary matrices $U$, $V$ such that
\begin{equation*}
|A+B| \leq U|A|U^{*}+V|B|V^{*}.
\end{equation*}
\end{lemma}
\begin{lemma}\cite[Theorem 2.1]{AB07}\label{3.3}
Let $f : [0,\infty) \to [0,\infty)$ be a concave function and $A$, $B \in \mathbb{M}_{n}^{+}$. Then there exist two unitary matrices $U$, $V \in \mathbb{M}_{n}$ such that
\begin{equation*}
f(A+B) \leq U^{*}f(A)U+V^{*}f(B)V.
\end{equation*}
\end{lemma}
A celebrated result due to Fan and Hoffman (see, e.g.\cite[p.73]{BR97}) is as follows.
\begin{lemma}\label{3-3}
Let $A \in \mathbb{M}_{n}$. Then
\begin{equation*}
\lambda_{j}(\Re A) \leq \sigma_{j}(A),\quad j=1,\cdots,n.
\end{equation*}
Equivalently, 
\begin{equation*}
\Re A \leq U|A|U^{*},
\end{equation*}
for some unitary matrix $U \in \mathbb{M}_{n}$.
\end{lemma}
\begin{lemma}\label{le2.6}
Let $A\ge B\ge 0$ and $f:[0,\infty)\to[0,\infty)$ be a concave function. Then
\begin{equation*}
		\|f(A )\|\ge \|f(B)\|.\end{equation*}
\end{lemma}
\begin{proof}
Note that  $f:[0,\infty)\to[0,\infty)$ is a concave function implies $f$ is nondecreasing. By using Weyl Monotonicity Theorem \cite[p. 63]{BR97}, we have $\lambda_j(f(A))=f(\lambda_j(A))\ge f(\lambda_j(B))=\lambda_jf((B)), 1\le j \le n$. Thus, $\|f(A )\|\ge \|f(B)\|.$
\end{proof}
It immediately follows from \ref{3-3} that
\begin{cor}\label{3.4}
	Let $A$ be a sectorial matrix and $f:[0,\infty)\to[0,\infty)$ be a concave function. Then
\begin{equation*}
		\|f(\Re A )\|\leq \|f(|A|)\|.\end{equation*}
\end{cor}

Next, we present our main results.
\begin{thm}\label{mt}
Let $f : [0,\infty) \to [0,\infty)$ be a concave function and let $A$ with $W(A) \subseteq S_{\alpha}$ for $\alpha \in [0,\pi/2)$ be partitioned as in (\ref{e1}). Then for every $s>0$,
\begin{equation*}
\Vert{f(|A|)}\Vert \leq \sum_{i=1}^2\left(\|f(\left|  A_{ii}\right| )\|+ \|f(\frac{s\tan\alpha}{2}\left|  A_{ii}\right| )\|+\|f(\frac{\tan\alpha}{2s}\left|  A_{ii}\right| \|\right).
\end{equation*} 
\end{thm}
\begin{proof}
Let the Cartesian decomposition of $A$ be $A = \Re A+i\Im A$ and $s>0$. Then from Lemma \ref{3.1}, we know that $\begin{bmatrix}
	\tan(\alpha)\Re A&\Im A\\
	\Im A&\tan(\alpha)\Re A
\end{bmatrix} \geq 0$.
So by Theorem \ref{t2.1}, we have
\begin{eqnarray}\label{e_1}
	\left|\Im A\right|\le \dfrac{s\tan\alpha}{2}\Re A+\dfrac{\tan\alpha}{2s}U^*\Re AU
\end{eqnarray}
for some unitary matrix $U\in \mathbb{M}_{n}$.
Thus,
\begin{eqnarray}\label{e_2}
|A| &=& |\Re A+i\Im A|\nonumber\\
&\leq& U_{1}\Re AU_{1}^{*}+V_{1}|\Im A|V_{1}^{*} \text{  (By Lemma \ref{3.2})}\nonumber\\
&\leq& U_{1}\Re AU_{1}^{*}+\frac{\tan\alpha}{2}V_{1}(s\Re A+s^{-1}U^*\Re AU)V_{1}^{*} \text{  (By (\ref{e_1}))}
\end{eqnarray}
for two unitary matrices $U_{1}$, $V_{1} \in \mathbb{M}_{n}$.

By Lemma \ref{3.3}, there exists three unitary matrices $U_{2}$, $V_{2}$, $V \in \mathbb{M}_{n}$ such that
\begin{eqnarray}\label{e_3}
	\begin{aligned}
		f(U_{1}\Re AU_{1}^{*}+\frac{\tan\alpha}{2}V_{1}(s\Re A+s^{-1}U^*\Re AU)V_{1}^{*}) &\leq U_{2}f(\Re A)U_{2}^{*}+V_{2}f(\frac{s\tan\alpha}{2}\Re A)V_{2}^{*}\\
		&+Vf(\frac{\tan\alpha}{2s}\Re A)V^{*}.
	\end{aligned}
\end{eqnarray}
Since $f : [0,\infty) \to [0,\infty)$ is concave, then $f$ is non-decreasing. Therefore,
\begin{eqnarray*}\label{e-0}
	\|f(|A|)\|&\leq& \|f(U_{1}\Re AU_{1}^{*}+\frac{\tan\alpha}{2}V_{1}(s\Re A+s^{-1}U^{*}\Re AU)V_{1}^{*})\| \text{  (By (\ref{e_2}) and Lemma \ref{le2.6})}\\
	&\leq& \|U_{2}f(\Re A)U_{2}^{*}+V_{2}f(\frac{s\tan\alpha}{2}\Re A)V_{2}^{*}+Vf(\frac{\tan\alpha}{2s}\Re A)V^{*}\|\text{  (By (\ref{e_3}) and Lemma \ref{le2.6}))}\\
	&\leq& \|f(\Re A)\|+ \|f(\frac{s\tan\alpha}{2}\Re A)\|+\|f(\frac{\tan\alpha}{2s}\Re A)\|\\
	&\le&\sum_{i=1}^2\left(\|f(\Re A_{ii})\|+ \|f(\frac{s\tan\alpha}{2}\Re A_{ii})\|+\|f(\frac{\tan\alpha}{2s}\Re A_{ii})\|\right)\\
	&&(\text{Since $\Re A = \begin{bmatrix}
			\Re A_{11}&\frac{A_{12}+A_{21}^{*}}{2}\\
			\frac{A_{12}^{*}+A_{21}}{2}&\Re A_{22}
		\end{bmatrix}\ge 0$, then by Theorem \ref{lt}})\\
 &\le&\sum_{i=1}^2\left(\|f(\left|  A_{ii}\right| )\|+ \|f(\frac{s\tan\alpha}{2}\left|  A_{ii}\right| )\|+\|f(\frac{\tan\alpha}{2s}\left|  A_{ii})\right| \|\right) \text{  (By Corollary \ref{3.4})}
	\end{eqnarray*}
\end{proof}
\begin{rem}
	When $s\in [1,2]$, Theorem \ref{mt} is stronger than Theorem \ref{zpt}. For instance, under same assumptions of Theorem \ref{zpt}, in Theorem \ref{mt}, for $s=1$, we have
	\[
	\Vert{f(|A|)}\Vert \leq \Vert{f(|A_{11}|)}\Vert+\Vert{f(|A_{22}|)}\Vert+2\left(\Vert{f(\frac{\tan\alpha}{2}|A_{11}|)}\Vert+\Vert{f(\frac{\tan\alpha}{2}|A_{22}|)}\Vert\right);
	\]
	for $s=2$, we have
\begin{eqnarray*}
		\Vert{f(|A|)}\Vert &\leq& \Vert{f(|A_{11}|)}\Vert+\Vert{f(|A_{11}|)}\Vert+\Vert{f(\tan\alpha|A_{11}|)}\Vert+\Vert{f(\tan\alpha|A_{22}|)}\Vert\\
		&&+\Vert{f(\frac{\tan\alpha}{4}|A_{11}|)}\Vert+\Vert{f(\frac{\tan\alpha}{4}|A_{22}|)}\Vert.
\end{eqnarray*}
\end{rem}

\begin{cor}
Let $f(t) = t^{p}$, $0 < p \leq 1$. For $s=1$ and let $A$ with $W(A) \subseteq S_{\alpha}$ for $\alpha \in [0,\pi/2)$ be partitioned as in (\ref{e1}). Then we have
\begin{equation*}
\||A|^{p}\| \leq (1+2^{1-p}(\tan\alpha)^{p})(\||A_{11}|^{p}\|+\||A_{22}|^{p}\|).
\end{equation*}
\end{cor}
Similarly, we can also get a generalization of Theorem \ref{t1.7} along a different path from Y. Mao in \cite{MJ24}.
\begin{thm}\label{m2}
	Let $f : [0,\infty) \to [0,\infty)$ be a concave function and let $A$ with $W(A) \subseteq S_{\alpha}$ for $\alpha \in [0,\pi/2)$ be partitioned as in (\ref{e1}). Then for any $s>0$, 
	\begin{equation*}
		\Vert{f(|A|)}\Vert \leq\sum_{i=1}^2 \left(\Vert{f\left(\frac{s\sec\alpha}{2}|A_{ii}|\right)}\Vert+\Vert{f\left(\frac{\sec\alpha}{2s}|A_{ii}|\right)}\Vert\right) .
	\end{equation*}
\end{thm}
\begin{proof} By Lemma \ref{3.1}, we know
	\[
	\begin{bmatrix}
		sec(\alpha)\Re A&A^{*}\\
		A&sec(\alpha)\Re A
	\end{bmatrix}\ge 0.
	\] By using Lemma \ref{t2.1}, it implies that
\begin{equation}\label{e4.1}
|A| \leq \frac{\sec\alpha}{2}(s\Re A+s^{-1}U^{*}\Re AU), \text{ for any $s>0$.}
\end{equation}
So by Lemma \ref{3.3}, there exists two unitary matrices $U_{1}$, $V_{1} \in \mathbb{M}_{n}$ such that
\begin{equation}\label{e4.2}
f(\frac{\sec\alpha}{2}(s\Re A+s^{-1}U^{*}\Re AU)) \leq U_{1}f(\frac{s\sec\alpha}{2}\Re A)U_{1}^{*}+V_{1}f(\frac{\sec\alpha}{2s}\Re A)V_{1}^{*}.
\end{equation}
Since $f : [0,\infty) \to [0,\infty)$ is concave, then $f$ is non-decreasing. Therefore,
\begin{eqnarray*}
\|f(|A|)\|&\leq& \|f(\frac{\sec\alpha}{2}(s\Re A+s^{-1}U^{*}\Re AU))\| \text{  (By (\ref{e4.1}) )  and Lemma \ref{le2.6})}\\
&\leq& \|U_{1}f(\frac{s\sec\alpha}{2}\Re A)U_{1}^{*}+V_{1}f(\frac{\sec\alpha}{2s}\Re A)V_{1}^{*}\| \text{  (By (\ref{e4.2}) ) and Lemma \ref{le2.6})}\\
&\leq& \|f(\frac{s\sec\alpha}{2}\Re A)\|+ \|f(\frac{\sec\alpha}{2s}\Re A)\|\\
&\le&\sum_{i=1}^2\left(\|f(\frac{s\sec\alpha}{2}\Re A_{ii})\|+ \|f(\frac{\sec\alpha}{2s}\Re A_{ii})\|\right)\\
&&(\text{Since $\Re A = \begin{bmatrix}
		\Re A_{11}&\frac{A_{12}+A_{21}^{*}}{2}\\
		\frac{A_{12}^{*}+A_{21}}{2}&\Re A_{22}
	\end{bmatrix}\ge 0$, then by Theorem \ref{lt}})\\
&\le&\sum_{i=1}^2\left(\|f(\frac{s\sec\alpha}{2}\left|  A_{ii}\right| )\|+ \|f(\frac{\sec\alpha}{2s}\left|  A_{ii}\right| )\|\right)\text{  (By Corollary \ref{3.4})}
\end{eqnarray*}
\end{proof}
\begin{rem}
	For $s=1$ in Theorem \ref{m2}, we obtain the result in \ref{ss}.
\end{rem}

 \section*{Acknowledgement} 
The authors thank professor M. Lin for introducing this topic and discussing with us heartily.

\end{document}